\documentclass{amsart}
\usepackage{amsmath}
\usepackage{amssymb}
\usepackage{graphicx}
\usepackage{color}
\usepackage[colorlinks]{hyperref}

 \newtheorem{thm}{Theorem}[section]
 \newtheorem{cor}[thm]{Corollary}
 \newtheorem{lem}[thm]{Lemma}
 
 \theoremstyle{definition}
 \newtheorem{defn}[thm]{Definition}
 \theoremstyle{remark}
 \newtheorem{rem}[thm]{Remark}
 \theoremstyle{example}
 
 \numberwithin{equation}{section}

 \newcommand{\F}{\mathcal{F}}

 \newcommand{\N}{\mathbb{N}}
 \newcommand{\Z}{\mathbb{Z}}

\begin{document}
\email{dahmadi1387@gmail.com, maliheh.dabbaghian@gmail.com}

\title[]
{Types of mixings and transitivities in topological dynamics}

\author[D. Ahmadi, M. Dabbaghian]{Dawoud Ahmadi Dastjerdi, Maliheh Dabbaghian Amiri}

\begin{abstract}
 Using the combinatorial properties of subsets of
integers, a classification of metric dynamical systems was
given in [V. Bergelson and T. Downarowicz, {Large sets of integers
and hierarchy of mixing properties of measure-preserving systems},
Colloquium Mathematicum {\bf 110}(1), (2008), 117-150]. As a result, some new families emerged.
Here, their counterparts in topological dynamics has been
considered. The differences will be discussed  and new classes of
systems in topological dynamics will be introduced. 
\end{abstract}
\maketitle
\section*{Introduction}
{
The parallel concepts for measure-theoretical dynamical systems (MDS) and topological
 dynamical systems (TDS) such as ergodicity vs minimality, measure-theoretical entropy
 vs topological entropy, measure-theoretical mixing vs topological mixing, etc. has been
 always  a subject for investigation \cite{glasner, huang-li, huang}.
In this respect, we first consider the different notions of mixing in TDS, that is, we partition the family of weak mixing TDS into different sub-classes  
 and will compare them to the similar ones considered for MDS in \cite{ber}; however, we extend this task to transitive TDS as well.  That is, we will also introduce different sub-classeses in transitive topological systems.
The main tool in \cite{ber} was combinatorial number theory which is one of ours and we also use spacing shifts to construct our examples. 
Spacing shifts, a class of subshifts, was first introduced by Lau and Zame in 1973 for constructing an example of a weak mixing system which is not strong mixing \cite{lau}. 
The fact that this can be done is due to the peculiarity of the spacing shifts that allows one to investigate dynamical properties of a subshift via combinatorial number theory. 
For instance,  Lau and Zame noticed that if $P$ is thick but not cofinite, then the associated spacing shift denoted by $\Sigma_P$ is weak mixing but not strong mixing. 
Our goal is more or less the same. 
We use families of subsets of integers, say from Hindman's table presented in {\cite{Hind}} and try to establish different examples of mixings and transitives in systems.
 An objective is to sort out if different families of integers do introduce different dynamics on the respective spacing shifts as well. The  families we are considering appear in diagram ~\eqref{diag}.

Our arrangement in this note is as follows.
The  Hindman's table is introduced in Section \ref{pre}.
The properties of a family $\mathcal{F}$ when  an $\mathcal{F}$-mixing system occurs is given in section \ref{mix-sec} and via examples by spacing shifts 
we show that they are really different mixings. 
Section \ref{tran-sec} is devoted to different transitive families using families of difference sets like $\mathcal{F-F}$ and we show that how these difference sets can be placed in Hindman's table.  
The application of our earlier results to minimal systems are considered in Section \ref{mds}. There we show that for minimals, the results very much resemble  the corresponding results in MDS.
}
\section{Preliminaries}\label{pre}
A TDS is a pair $(X,\,T)$ such
that $X$ is a compact metric space and $T$ is a
homeomorphism.
  In a TDS,
the \emph{return times set} is defined to be $N(U,\,V)=\{n\in
\mathbb {\mathbb Z}: T^n(U)\cap V\neq \emptyset\}$ where $U$ and
$V$ are \emph{opene} (non-empty and open) sets. A TDS is
\emph{transitive} if for any two opene sets $U$ and $ V$, 
$N(U,\,V)\neq \emptyset$; and it is \emph{totally transitive} if $(X,\,T^n)$ is transitive for any $n$. A TDS 
 is \emph{weak mixing}
if the product system $(X\times X,\, T\times T)$ is transitive and it
 is \emph{(strong) mixing} if $N(U,\, V)$ is cofinite for any
opene sets $U,\, V$.

Let ${\mathcal F}$ be a family of nonempty subsets of $\mathbb Z$.
This means
 ${\mathcal F}$ is hereditary upward: if $F_1\in  {\mathcal F}$ and $F_1\subseteq F_2$, then  $F_2\in{\mathcal F}$.
The \emph{dual} of  ${\mathcal F}$, denoted by ${\mathcal F}^*$,
is defined to be all subsets of $ \mathbb Z$
  meeting all sets in ${\mathcal F}$:
$${\mathcal F}^*= \{G\subset \mathbb Z :\ G\cap F\neq
\emptyset, \ \forall F \in {\mathcal F} \}.$$

A family $\mathcal{F}$ is called \emph{partition regular} if
$F\in\mathcal{F}$  is partitioned into finite sets
$F=F_1\cup\cdots\cup F_k$, then there is $i$ such that
$F_i\in\mathcal{F}$. A non-empty family closed under finite
intersections is called a \emph{filter}. It is known that if
$\mathcal{F}$ is partition regular, then $\mathcal{F}^*$ is a
filter.  A filter which is partition regular is called an
\emph{ultrafilter}. The collection of all ultrafilters is denoted
by $\beta\mathbb Z$ and when endowed with appropriate topology,
becomes the \emph{Stone-$\breve{C}$\!eck compactification} of
$\mathbb Z$. There is a natural semigroup structure in
$\beta\mathbb Z$ extending the addition operation of $\mathbb Z$
\cite{idem}. An ulterafilter $p\in\beta\mathbb Z$ is called
\emph{idempotent} if $p+p=p$.

 For a
family $\mathcal{F}$ and $k\in \mathbb Z$, the \emph{shifted
family} is defined as ${\mathcal F}+k=\{F+k:\
F\in{\mathcal{F}}\}$ where $F+k=\{n+k:\ n\in F \}$. An example of a shift invariant is 
$\mathcal I ^*$,  the family of all cofinite sets ($\mathcal I =\{F\subset \mathbb Z: \ |F|=\infty\}$). 
 Also, starting with a family  $\mathcal F$,
families  $\mathcal F_+$  and $\mathcal F_{\bullet}$ defined as
$$\mathcal F_+:=\bigcup_{k\in
\mathbb Z}({\mathcal F}+k),\qquad\mathcal F_{\bullet}:=
\bigcap_{k\in \mathbb Z}({\mathcal F}+k)$$
are shift invariant families \cite{ber}.

We have $\mathcal F_{\bullet}\subset \mathcal F \subset \mathcal
F_+$ and  $\mathcal{F}^*_{\bullet}=(\mathcal F_+)^*$
\cite{ber}.
{
  Also,
  if $\mathcal{F}\subset
\mathcal{F}'$ then
 $\mathcal{F}_+\subset \mathcal{F}'_+$ and
 $\mathcal{F'}^*\subset \mathcal{F}^*$ which
 implies that
$\mathcal{F'}^*_{\bullet}\subset \mathcal{F}^*_{\bullet }$.
}
 If
$\mathcal{F}$ is a filter, so is any shift of $\mathcal{F}$ and
since the intersection of filters is again a filter,
$\mathcal{F}_\bullet$ is a filter.

 A set $F \subset \mathbb Z$ is called an \emph{$IP$-set} if it
contains the set of  finite sums of some sequence of nonzero integers $A=\{a_n\}_{n \geq 1}$ and is denoted by $ FS(A)$; thus,  $ FS(A)=\{a_{i_1}+a_{i_2}+\cdots+a_{i_n}, \ \
i_j<i_{j+1}\}\subset F$.  Equivalently, a
subset $F\subset \mathbb Z$ is an $IP$-set if and only if there
is an idempotent $p\in \beta\mathbb Z$ such that $F\in p$
\cite[Theorem 1.5]{idem}. Denote by ${\mathcal{IP}}$ the family of
all $IP$-sets.

Set $d^*(A):=\limsup_{(M-N)\rightarrow\infty}\frac{|A\cap
[N,\,M]|}{M-N+1}$ and call it the \emph{upper Banach density} of a set
$A\subset \mathbb Z$. We denote the family of all positive
upper Banach density by $\mathcal{F}_{pubd}$.  Also, the
\emph{lower Banach density} of a set $A\subset \mathbb Z$ is
defined similarly as $d_*(A)=\liminf_{(M-N)\rightarrow\infty}\frac{|A\cap
[N,\,M]|}{M-N+1}$.
 
A subset $E\subset\mathbb Z$ is called
\begin{itemize}
\item \emph{$\Delta$-set} if there exists a sequence of natural numbers
$S=(s_n)_{n\in\mathbb N}$ such that the difference set
$\Delta(S)=\{s_i-s_j; \ \ i>j\} \subset E$.  Let ${\Delta}$ be
the family of all $\Delta$-sets
  \item \emph{thick}, if it contains arbitrarily long intervals
$[a,\,b]=\{a,\,a+1,\,\cdots,\,b\}\subseteq \mathbb{Z}$.  We denote the family of all
thick sets by $\mathcal{T}$.
  \item \emph{syndetic,} if its complement is not a thick set.
  Equivalently, it has  bounded gaps.  We have $\mathcal{S}={\mathcal{T}}^*$.
  \item \emph{piecewise syndetic}, if it is the intersection of a
  thick set and a syndetic set and the corresponding family is denoted by $\mathcal{PS}$.
  \item  \emph{thickly  syndetic} if its complement is not piecewise syndetic  and  denote by $\mathcal{TS}$, the family of thickly syndetic sets. 
We have  $\mathcal{TS}= \mathcal{PS}^*$.
  \item $D$-\emph{set} if it is a member of an idempotent $p$ whose any
member has positive upper Banach density;  the
family of all $D$-{sets} is denoted by $\mathcal D$.  
\item \emph{central} or \emph{C-set} if it is a member of a
minimal idempotent of
 $\beta\mathbb Z$ and let $\mathcal C$ be the
family of all $C$-sets.
\item \emph{Bohr set} if there exist $m\in \N,\, \alpha \in \mathbb{T} =\{z\in \mathcal{C} : |z| = 1\}$ and open set $U\subset \mathbb{T}^m$
 such that the set $\{n \in \Z : n\alpha \in U\}$ is contained in $E$.

\end{itemize}

{


\section{Mixing properties}\label{mix-sec}
Let $\mathcal{P}(\Z)$ be the power set of $\Z$.
  For any family $\mathcal{F}$ 
$$\tau\mathcal{F}=\{F\in\mathcal{P}(\Z): (F-k_1)\cap\cdots\cap(F-k_n)\in\mathcal{F}, \{k_1,\ldots,k_n\}\subset \Z\}$$
is defined to be the \emph{thick family} contained in $\mathcal{F}$. A family $\mathcal{F}$ is \emph{thick family} if and only if
 $\tau\mathcal{F}=\mathcal{F}$  \cite{huang3}. 

\begin{thm}\label{thick-filter}
  $\tau \mathcal{F}\subset \mathcal{F}_{\bullet}$ and if $\mathcal{F}$ is a filter, then $\tau \mathcal{F}=\mathcal{F}_{\bullet}$.
\end{thm}
\begin{proof}
 $\tau \mathcal{F}\subset \mathcal{F}_{\bullet}$ follows trivially from definition. For the second part, since $\mathcal{F}$ is a filter,  $\mathcal{F}_{\bullet}$ is a filter as well.
 Let $F\in \mathcal{F}_{\bullet}$ and pick $k_1,\,k_2,\,\ldots,\,k_n\in \mathbb{Z}$. 
Since $\mathcal{F}_{\bullet}$ is a shift invariant family and
 it is closed under finite intersections, by the filter property, we have $(F-k_1)\cap \cdots \cap (F-k_n)\in\mathcal{F}_{\bullet}$.
 Hence $F\in \tau\mathcal{F}$ and $\tau\mathcal{F}=\mathcal{F}_{\bullet}$. 
\end{proof}
The classical different mixing concepts for a TDS are
(strong) mixing, mild mixing and weak mixing. They can be
characterized in terms of combinatorial properties of ${\mathcal{N}}$, { the family which is generated by return times set of $(X,\,T)$}. A TDS is mixing, mild mixing or
weak mixing if and only if $\mathcal{N}$ is cofinite, $
(\mathcal{IP-IP})^* $ or thick respectively \cite{huang}. 

 A TDS $(X,\,T)$ is called $\mathcal F$-\emph{transitive} if for any
two opene sets $U,\,V \subset X$,  $N(U,\,V)\in {\mathcal
F}$ and it is  called \emph{$\mathcal F$-mixing} if  the product
system  $(X\times X,\,T\times T)$ is  ${\mathcal F}$-transitive \cite{huang3}.
It is easy to see that
  $(X,\,T)$ is $\mathcal F$-mixing if and only if
  it is weak mixing and $\mathcal F$-transitive if and only if
    $N(U,\,V) \cap N(U,\,U)\in \mathcal F$ for any two opene
  sets $U, V$.

Now we may explicitly define  other mixings
 by considering the combinatorial  structure of $\mathcal{N}$. A motivation for this is the following
hierarchy in the combinatorial number theory:
\begin{eqnarray}\label{mixingfamily}
{\mathcal I}^* \ \subset \  { \Delta}^* \ \subset \  \mathcal{
IP}^* \ \subset  {\mathcal D}^* \ \subset  {\mathcal C}^*.
\end{eqnarray}
As an application, we will define
$\Delta^*,\,\mathcal{IP}^*,\,{\mathcal D}^* $ and ${\mathcal
C}^*$-mixings and by constructing examples, the fact that they
are different families will be shown. First a theorem:

\begin{thm}\label{mixing} 
\begin{enumerate}
\item A TDS
 $(X,\,T)$ is $\mathcal{F}$-transitive if and only if it is ${\mathcal
F}_{\bullet}$-transitive.
\item If $\mathcal{F}$ is a filter,
then $(X,\,T)$ is $\mathcal F$-mixing if and only if it
is ${\mathcal F}_{\bullet}$-mixing.
\item  $(X,\,T)$ is $\mathcal F$-mixing if and only if it
is ${\tau\mathcal F}$-mixing.
\end{enumerate}
\end{thm}
\begin{proof}
See \cite[Theorem 2.9]{our2} for (1) and (2). Also, a similar proof  for (2) in \cite[Theorem 2.9]{our2} gives (3).
\end{proof}

By definition, $\mathcal{I}^*$ is a filter. Also,
  ${\Delta},\,\mathcal{IP},\,\mathcal{D}$ and
$\mathcal{C}$ are union of ultrafilters \cite{ber}, so their dual families
are filters. Therefore, all the families appearing
 in \eqref{mixingfamily} are filters and  $\mathcal{F}$-mixing can be defined for them.

 A \emph{homomorphism} (or a \emph{factor map}) $\pi: (X,\,T)\to (Y,\,S)$ is a continuous onto
map from $X$ to $Y$ such that $S\circ \pi=\pi\circ T$.  Then $(X,\,T)$ is said to be
an \emph{extension} of $(Y,\,S)$ and $(Y,\,S)$ is called a \emph{factor} of $(X,\,T)$. The extension $\pi$ is called \emph{almost one-to-one} if the set 
$\{x\in X: \pi^{-1}(\pi(x))=\{x\}\}$ is a dense $G_{\delta}$ subset of $X$ \cite{AGHSY}.
 
 In \cite[Theorem 2.10]{our2}, it is shown that any non-trivial factor of an $\mathcal{F}$-mixing system is also $\mathcal{F}$-mixing  
 {and in \cite[Lemma 3.6]{huang5}, it is shown that the almost 1-1 extension of an $\mathcal{F}$-transitive is $\mathcal{F}$-transitive.
  In fact, in \cite[Lemma 3.6]{huang5}, the authors show that for any  almost 1-1 factor map $\pi:(X,\,T)\to (Y,\,S)$ if we
  let $\mathcal{N}_T$ and $\mathcal{N}_S$ the family which is generated by return times set of $(X,\,T)$ and $(Y,\,S)$ respectively,
 then $\mathcal{N}_T=\mathcal{N}_S$. Now if $(Y,\,S)$ is an $\mathcal{F}$-mixing system, then by Theorem \ref{mixing}(3),  
 $\mathcal{N}_S$ is a thick family in $\mathcal{F}$ and by $\mathcal{N}_T=\mathcal{N}_S$ we have $\mathcal{F}$-mixing lifted to an almost one-to-one extension.}

\subsection{Spacing shifts}
Let $A=\{0,1\}$ and $\sigma: A^{\mathbb Z}\rightarrow
A^{\mathbb Z}$ be a shift map. Then $(A^{\mathbb Z},\, \sigma)$ is called \emph{the space of full shift on two symbols} and any closed subset
$\Sigma\subset A^{\mathbb Z}$ which is invariant under $\sigma$ is a
\emph{subshift} of $A^{\mathbb Z}$. A \emph{word}  $u=u_1u_2\cdots u_l$ appears in $x=\{x_i\}_{i\in\mathbb Z
}\in \Sigma$ at position $t$ if $x_{t+j-1}=u_j$ for $j=1,\cdots,
l$.  The set of all finite words that appear in some elements of
$\Sigma$ together with empty word $\epsilon$ is called the
\emph{language} of $\Sigma$ and is denoted by
$\mathcal{L}(\Sigma)$.

Now we recall the spacing shifts which are a valuable class in TDS providing examples for
different types of mixings.

\begin{defn}\label {isp}
Let $P$ be a subset of integers with $P=-P$.
 Define an \emph{invertible spacing shifts} to be the  subshift
 \begin{equation}\label{zsp}
\Sigma_{P}^{\mathbb Z} =\{s \in \{0,\,1\}^{\mathbb Z} : \
s_i=s_j=1  \ \Rightarrow  i-j \in P \cup
\{0\}\}.
\end{equation}
It follows from the definition that $P=N([1],\,[1])$.
\end{defn}

\begin{rem}
Definition \ref{isp} is an invertible version of the classical
spacing shifts which are defined for $P\subset \mathbb N$. We will denote them by $\Sigma_P^{\mathbb N}$ and then
\eqref{zsp} changes to
$$\Sigma_P^{\mathbb N}=\{s \in \{0,\,1\}^{\mathbb N}: \  s_i=s_j=1  \
\Rightarrow  |i-j| \in P\cup \{0\} \}.$$ See \cite{sp} for a
detailed study of $\Sigma_P^{\mathbb N}$. Note that if
$P^+:=\{p\in P: p\geq 0\}$ and $P^-:=-P^+$ then $P=P^+\cup P^-$
and $\mathcal{L}(\Sigma_P^{\mathbb Z})=
\mathcal{L}(\Sigma_{P^+}^{\mathbb N})$.
\end{rem}

\begin{lem}\label{NUV}
Suppose $U,\,V \in \mathcal L(\Sigma_P^{\mathbb Z})$. Then  there exist $k_1,\, k_2, \ldots,\,
k_n\in\mathbb Z$ such that $N(U,\,V)\supseteq \bigcap_{i=1}^{n}(P
- k_i)$.
\end{lem}

\begin{proof}
If either $U$ or $V$ is $0^i$ then $N(U,V)=\mathbb Z$. Otherwise,
we may assume
\begin{eqnarray}\label{UV}
&& U=0^{a_1}10^{p_1-1}10^{p_2-1}1\cdots10^{p_r-1}10^{b_1} \nonumber\\
&& V=0^{c_1}10^{q_1-1}10^{q_2-1}1\cdots 10^{q_t-1}10^{d_1}
\end{eqnarray}
where  $p_i,\, q_j  \in  P^+$  for $1\leq i \leq r, \,1\leq j \leq
t $ and
 $a_1, \,b_1, \,c_1,\, d_1 \in \mathbb N \cup \{0\}$ and there are $A,\, B$ and $C$ such that $N(U,\,V) = A \cup B
\cup C $ where
\begin{eqnarray*}
&& A=\left\{ |W| : \ VWU\in \mathcal L(\Sigma_P^{\mathbb Z})\right\} +|V|, \nonumber\\
&& B=\left\{-|W|:  \ UWV\in \mathcal L(\Sigma_P^{\mathbb Z })\right\}-|U|, \nonumber\\
&& |C|<\infty \,\,\,(|C|\leq |V|+|U|).
\end{eqnarray*}
In particular, $A\cup B\subseteq N(U,\, V)$.

Let $W\in \mathcal L(\Sigma_P^{\mathbb Z})$. If we replace some of
the 1's in $W$ with zero's and calling the new word $W'$, then
$W'\in \mathcal L(\Sigma_P^{\mathbb Z})$. Therefore, we may let
$$A= \{n\in\mathbb N: V0^nU\in\mathcal L(\Sigma_P^{\mathbb Z})\}+|V|$$
and the same for $B$. Suppose $n\in A$.  Then by the way $V$ and
$U$ are presented in (\ref{UV}), we must have $d_1+n+a_1\in
P^+$, that is, $n\in P^+-(d_1+n_1)$.  Note that $d_1+n+a_1$ is
the length of space between the last 1 in $V$ and the first 1 in $U$. In
fact, $d_i+n+a_j\in P^+$ where $d_i$ is the position of $i$'th 1
from right on $V$ and $a_j$ is the position of the $j$'th 1 from
left in $U$. Hence, $n\in P^+-(d_i+a_j)$ where $ 1\leq i \leq t,
\, 1\leq j\leq r$. Let $\{l_1,\cdots, l_m\}=\{d_i+a_j+|V|: \,
1\leq i \leq t, \,1\leq j\leq t \}$ for appropriate $m$. This
means that $A=\cap_{i=1}^{m}(P^+-l_i)$. Similarly,
$B=\cap_{i'=1}^{m'}(P^--l'_{i'})$ and the proof is complete by
setting $\{k_1,\cdots, k_n\}=\{l_i: 1\leq i \leq m\}\cup
\{l'_{i'}: 1\leq i'\leq m'\}$.
\end{proof}

\begin{cor}\label{tran-sigma}
If $P\in\tau\F$ then $\Sigma_P^{\mathbb Z}$ is $\F$-transitive. 
\end{cor}

Combinatorial properties of $P$ could lead to dynamical
properties for $\Sigma_P^{\mathbb N}$ or $\Sigma_P^{\mathbb Z}$.
For instance, $\Sigma_P^{\mathbb N}$ is mixing if and only if $P$
is cofinite and it is weak mixing if and only if $P$ is thick
\cite{sp}. The same is true for $\Sigma_P^{\mathbb Z}$. 

\begin{thm}\label{Fmixing}
   $\Sigma_P^{\mathbb Z}$
is $\mathcal F$-mixing if and only if $P \in \tau
{\mathcal  F}$.
\end{thm}
\begin{proof}
 The necessity is a  consequence of  Theorem
 \ref{mixing}.
For the converse,  $P \in {\tau\mathcal  F}$ if and only if
any intersection of finite shifts of $P$ is in $\mathcal F$.
Applying Lemma~\ref{NUV},  $N(U,\,V)\supseteq
\bigcap_{i=1}^{n}(P - k_i)$ and $N(U,\,U)\supseteq
\bigcap_{i=1}^{m}(P - k'_i)$ for some $k_i,\,k'_j$.  So  $N(U,V)\cap N(U,\,U)$ is in
$\mathcal F$ which means $\Sigma_P^{\mathbb Z}$ is $\mathcal
F$-mixing.
\end{proof}
Therefore, if
$P\in\{\mathcal{I}^*_\bullet,\,\Delta^*_\bullet,\,\mathcal{IP}^*_\bullet,\,\mathcal{D}^*_\bullet,\,\mathcal{C}^*_\bullet\}$,
then  $\Sigma_P^{\mathbb{Z}}$ is $\mathcal{F}$-mixing for the
respective $\mathcal{F}\in
\{\mathcal{I}^*,\,\Delta^*,\,\mathcal{IP}^*,\,\mathcal{D}^*,\,\mathcal{C}^*\}$.
The classical mixing is the $\mathcal{I}^*$-mixing, so
we just say mixing instead of $\mathcal{I}^*$-mixing.

 In \cite{ber}, it has been shown that not all the mixings defined for the families in
$\{\mathcal{I}^*,\,\Delta^*,\,\mathcal{IP}^*,\,\mathcal{D}^*,\,\mathcal{C}^*\}$
are different in the measure theoretic case. See also   section~\ref{mds} for the equivalencies of some of these mixings for certain classes in TDS.
 However, as
one may expect, they are different families in a general TDS as the next
result shows.
\begin{thm}\label{F^*-mixing}
The following hierarchy is proper.
\begin{equation}\label{FirstRow}
\text{mixing}\subset\Delta^*\text{-mixing}\subset\mathcal{IP}^*\text{-mixing}\subset\mathcal
D^*\text{-mixing} \subset\mathcal{C}^*\text{-mixing}\subset\text{weak mixing} 
\end{equation}
\end{thm}
\begin{proof}
We choose our examples from spacing shifts. First, we show that
there is a $\Delta^*$-mixing spacing shifts which is not 
mixing.
 Let $E=\{n^2: n\in\mathbb N\}$ and note that $E$ is not a
 $\Delta$-set and set $P^+=\mathbb N\setminus E$. Then $P$ is $\Delta^*_{\bullet}$, and as a result
$\Sigma_P^\mathbb{Z}$ is $\Delta^*$-mixing.  However, it cannot be  mixing
because $P$ is not cofinite.

For the remaining cases, recall first that there are examples of
subsets of integers which are $ C^*_{\bullet}$ but not $D^*_+$, $
D^*_{\bullet}$ but not $ IP^*_+$ and $ IP^*_{\bullet}$ but not $
\Delta^*_+$ \cite{ber}.  This  implies that the associated
spacing shifts are different.

To show that there exists weak mixing system which is not $\mathcal{C}^*$-mixing 
notice that   $\mathcal{PS}=(\mathcal{C}_+)$ \cite{ber}, thus
\begin{equation}\label{TS}
\mathcal{C}^*_{\bullet}=(\mathcal{C_+})^*=
(\mathcal{PS})^*=\mathcal{TS}.
\end{equation}
 In particular,  $\mathcal{TS}$  is a filter. Now if we pick $E\subset \Z$ so that both $E$ and $E^c$ are thick sets, then $E$ is not $C^*$ and therefore, $\Sigma_P^{\Z}$ is a weak mixing which is not $\mathcal{C}^*$-mixing.
\end{proof}
{
Let us recall that in MDS mixing and $\Delta^*$-mixing in one side and $\mathcal{D}^*$-mixing and weak mixing (and so $\mathcal{C}^* $-mixing as well) in the other side were the same \cite{ber}.

By definition a TDS is mild mixing iff it is $(\mathcal{IP-IP})^*$-transitive. This motivates us 
to consider the following diagram, an extension of \eqref{FirstRow},  for achieving a 
 possible classification in  TDS via these mixings and transitivities concepts. 
}
\begin{eqnarray}\label{diag}
\begin{array}{ccccccccccc}
&\!\!\!\!\!\!  \mathcal{I}^*  &\!\!\!\!\subset &\!\!\!\! {\Delta}^*
&\!\!\!\! \subset  &\!\!\!\! \mathcal{IP}^*  &\!\!\!\! \subset
&\!\!\!\!  \mathcal{D}^*  &\!\!\!\!  \subset &\!\!\!\!
 \mathcal{C}^*&\\
&\!\!\!\!\!\cap&\!\!\!\! \ &\!\!\!\!\cap&\!\!\!\!\ &\!\!\!\!  \cap &\!\!\!\!\ &\!\!\!\! \cap &\!\!\!\!\ &\!\!\!\!\! \cap \!\\
\ \ \ \ \ \Delta^*=&\!\!\!\! (\mathcal{I-I})^* &\!\!\!\! \subset  &\!\!\!\!
({\Delta-\Delta})^* &\!\!\!\! \subset &\!\!\!\!
 (\mathcal{IP-IP})^*&\!\!\!\! \subset &\!\!\!\!  (\mathcal{D-D})^* &\!\!\!\! \subset &\!\!\!\!
 (\mathcal{C-C})^*.
 \\
\end{array}
\end{eqnarray}

We showed in Theorem \ref{Fmixing} that how different mixings can be defined for the first
row. The families in the second row are not necessarily filters
and we  may only define  transitivity for them.
The above list in \eqref{diag} is rather complete with respect to 
 Hindman's table \cite{Hind}; for as we will point out later, $(\mathcal{S-S})^*=(\mathcal{C-C})^*=(\mathcal{PS-PS})^*$ and  
$( \mathcal{F}_{pubd}-\mathcal{F}_{pubd})^*=(\mathcal{D-D})^*$.
  Now  the first
problem to be set is if they actually represent different
families in TDS; though from set theoretical point of view, they are all different families. Recall that  whether there is a
$(\mathcal{C-C})^*$-transitive system which is not $(\mathcal{D-D})^*$-transitive is
yet an open problem first raised in  \cite{huang4}.
 Except this
case, we will show in sequel that all other families in \eqref{diag} represent
different transitive TDS.
One case is already known: there exists $(\mathcal{IP-IP})^*$-transitive which is not $\mathcal{IP}^*$-mixing \cite[Example C]{huang3}.
Other cases are treated in Theorem~\ref{filter} and Lemma~\ref{C-C}.

We also look for the thick families inside the families in the
second row in \eqref{diag}. By Theorem~\ref{mixing}, these thick families define some other mixings.
 For instance, any thick families in the
$(\mathcal{IP}-\mathcal{IP})^*$-transitive family will be a mild
mixing system. Then  we   investigate how these new mixing families
are different from those on the first row.

\section{ Characterization of $(\mathcal{F-F})^*$-transitive and
$\mathcal{F}^*$-mixing}\label{tran-sec}

An increasing set  $\{a_n\}_{n\in \N}$
is said to have \emph{progressive gaps} if it contains a subsequence $\{a_{n_k}\}$ (call each finite subset $\{a_{n_k+1},\, a_{n_k+2},\,\ldots,\, a_{n_{k+1}}\}$ a \emph{chunk})
 such that for $n_k+1 < i \leq n_{k+1}$ one
has $a_i- a_{i-1} > a_{n_{k+1}}- a_i$ (inside each chunk every gap is larger than the
 distance to the right end of the chunk) and $a_{n_{k+1}}-a_{n_k}\to\infty$ (the gaps
 between the chunks tend to infinity) \cite{ber}.


 \begin{lem}\label{pdel}
There exists a set $B\subset \N$ such that $\Delta(B)$ is not a
$(\Delta-\Delta)$-set and $(\Delta(B)-\Delta(B))$ is not an  $IP$-set.
\end{lem}
\begin{proof}
Suppose $B=\{b_n\}_{n\in \mathbb N}$ is a rapidly growing sequence of
natural numbers, for instance, let  
$b_n>4 \sum_{i=1}^{n-1}b_i$ ($b_n>4 \sum_{i=1}^{n-1}(b_i+i)$ was considered in  \cite[Example C]{huang3}).
Let $A=\Delta(B)$ and observe that $A$ has progressive gaps. Also, for a certain distance $d$, $A_d=\{a_{i_j}\in A: a_{i_j+1}-a_{i_j}=d \}$
is either  finite or  $\lim_{j\to\infty}(a_{i_{j+1}}-a_{i_j})=\infty$.  This is implied by the fact that   the distance $d$ between any two elements must eventually occur only inside the chunks of $A$ and
then only once in every chunk. 

To show that $\Delta(B)$ is not a
$(\Delta-\Delta)$-set, it suffices to show  that  for any arbitrary increasing sequence $S\subset \N$,  $L=(\Delta(S)-\Delta(S))$ cannot have progressive gaps.

First, 
since $ s_j-s_i=(s_j-s_k)-(s_k-s_i)\in L$, $\Delta(S)\subset L$. 
Now pick any $r,\,t,\,j$; $r>t>j$ and note that for any $n>r$
    we have $(s_n-s_r),\,(s_n-s_t),\,(s_n-s_j)$ and $(s_n-s_j)-(s_r-s_t)$  in $L$ and  
\begin{equation}\label{delta-delta}
d=(s_n-s_j)-(s_n-s_t)=[(s_n-s_j)-(s_r-s_t)]-(s_n-s_r).
\end{equation}
This means that for any $n$,  $\lim_{n\to\infty}([(s_n-s_j)-(s_r-s_t)]-(s_n-s_t))=(s_t-s_j)\neq\infty$ and this in turn  shows that
 $L$ does not have progressive gaps.

For the second part assume that $(\Delta(B)-\Delta(B))$ contains an $IP$-set $F\subseteq\N$. Since $\Delta(B)\subset (\Delta(B)-\Delta(B))$ and  $\Delta(B)$ has 
 progressive gaps,  infinitely many members of $F$ are in $(\Delta(B)-\Delta(B))\setminus \Delta(B)$. 
This enables us to choose $f_1,\,f_2\in F\cap ((\Delta(B)-\Delta(B))\setminus \Delta(B))$ 
so that if 
$f_1=(b_{r_1}-b_{r_2})-(b_{r_3}-b_{r_4}),\,f_2=(b_{r_5}-b_{r_6})-(b_{r_7}-b_{r_8})$, then  $b_{r_{i+1}}<b_{r_i}$, $1\leq i\leq 7$. 
Let $f_3=f_1+f_2=(b_{t_1}-b_{t_2})-(b_{t_3}-b_{t_4})\in F$. 
Then by the way $B$ has been choosen, $r_1=t_1,\,\{r_2,\,r_3\}=\{t_2,\,t_3\}$ and $r_4=t_4$. This means that 
 $f_2$ must be zero which is absurd. 
\end{proof}

Using the concept of chunks,  Bergelson--Downarowicz
 prove that there exists an $IP^*_{\bullet}$-set which is not $\Delta^*_+$ \cite[Theorem 2.11]{ber}. To prove this fact, they show that any $IP$-set, likewise any  $(\Delta-\Delta)$-set,
 cannot have  progressive gaps. Then they consider a set $E$  as a union over all integers $k$ of shifted by $r_k$ of $\Delta$-sets $E_k$ such that $E$ has progressive gaps, 
and as a result it cannot contain any  shifted $IP$-set.
 This argument shows that $E$ does not have any $(\Delta-\Delta)$-set either. 
Now, the
 complement of such $E$ is $(\Delta-\Delta)^*_{\bullet}$ and not $\Delta^*_+$. See \cite[proof of Theorem 2.11]{ber} for the construction of $E$.
 So we have:
\begin{lem}\label{our}
 There exists a ${(\Delta-\Delta)}^*_{\bullet}$-set which is not
${\Delta}^*_{\bullet}$-set.
 \end{lem}

\begin{thm}\label{filter}
\begin{enumerate}
\item
  There exists a $(\Delta-\Delta)^*$-transitive  which is not
$\Delta^*$-mixing.
\item
There exists an $(\mathcal{IP-IP})^*$-transitive which is
 not  $(\Delta-\Delta)^*$-transitive. 
\end{enumerate}
\end{thm}
\begin{proof}
(1)
We show that $(\Delta-\Delta)^*_{\bullet}$ contains a thick family $\mathcal{F}$
different from $\Delta^*_{\bullet}$. Then by applying Corollary~\ref{tran-sigma},
$\Sigma^\mathbb{Z}_P$ will be the required space.

 Let $E$ be the set given in Lemma~\ref{our} and set
$P^+:=E^c\subset \mathbb N,\,P:=P^+\cup
P^-$.
 Also let  $$\mathcal{F}:=\{(P+ k_1)\cap(P+
k_2)\cap\cdots\cap(P+ k_n):\, \{k_1,\,k_2,\,\ldots,\,k_n\}\subset\mathbb{Z}\}.$$ 
One needs  to show that any element of
$\mathcal{F}$ is a $(\Delta-\Delta)^*_{\bullet}$-set and to have this
we prove that the complement of an element of $\mathcal{F}$
does not contain a $(\Delta-\Delta)$-set. Since { $(P+
k_1)  \cap \cdots\cap(P+ k_n)=\left((P^++ k_1) \cap  \cdots\cap(P^++ k_n)\right) \cup \left((P^-+ k_1) \cap \cdots \cap (P^-+ k_n)\right)$}, $\mathcal{F}$ will be a thick family in $(\Delta-\Delta)^*$
  if
we show that $(P^++ k_1)\cap\cdots\cap(P^++
k_n)$ is a $(\Delta-\Delta)^*$-set or equivalently
$(E+k_1)\cup\cdots\cup(E+k_n)$
does not contain any $(\Delta-\Delta)$-set.

We show this fact for  $E\cup (E+k)$, general case follows similarly. Assume the contrary and let $L$ be a $(\Delta-\Delta)$-sets  such that $L\subset E\cup (E+k)$.
For any $i\in\N$, choose  $\{l_1^i,\,l_2^i,\,l_3^i,\,l_4^i\}_{i\in\N}$ satisfying \eqref{delta-delta}, i.e.  $l_2^i-l_1^i=l_4^i-l_3^i=d$ and observe that for all $i$, $l_3^i-l_1^i=l_3^1-l_1^1$. We may assume that $l_3^1-l_1^1>k$.
Clearly $l_1^i,\,l_2^i,\,l_3^i,\,l_4^i$ cannot be in $E$ or $E+k$.
Therefore, suppose for some $i$, $l_1^i\in (E+k)$ and  $l^i_j\in E,\,2\leq j\leq 4$ (similar argument applies for other cases). Then $l_1^i-k$ is in  the same chunk of $E$  containing  $l_j^i$'s. But one has $l_2^i-(l_1^i-k)=k+d<l^i_4-l^i_2$. This is absurd since $E$ has  progressive gaps. 

(2)
Choose $B$ satisfying Lemma~\ref{pdel} and set $P^+:=(\Delta(B)-\Delta(B))^c\subset \N$; then $P=P^+\cup P^-$ is $\mathcal{IP}^*_{\bullet}$ which is not  $(\Delta-\Delta)^*$. 
So the associated $\Sigma_P^{\Z}$ is $\mathcal{IP}^*$-mixing but  not $(\Delta-\Delta)^*$-transitive. 
Now since any  $\mathcal{IP}^*$-mixing is  $(\mathcal{IP-IP})^*$-transitive, $\Sigma_P^{\Z}$ must be $(\mathcal{IP-IP})^*$-transitive. 
\end{proof}


For any $\epsilon>0$ there exists $A\subset \mathbb{N}$ such that $d(A)>1-\epsilon$ 
and  there is  not any sequence $\{x_n\}_{n\in\mathbb{N}}$ in
 $\mathbb{N}$ and $t\in \mathbb{Z}$ such that $t+ FS(\{x_n\}_{n\in\mathbb{N}})\subset A$ \cite[Theorem 2.20]{ber2}.
 This means that there is a set with positive density which  is not a $D$-set  and so $\mathcal{D}$ and $\mathcal{F}_{pubd}$ are 
actually different families. However, in the next result, we will show that $\Delta(\mathcal{D})=\Delta(\mathcal{F}_{pubd})$.
Before that we need some more terminologies.

Let $\mathcal{F}$ be a family. The \emph{block family} of $\mathcal{F}$, denoted by $b\mathcal{F}$, is the family
consisting of sets $F\subset \mathbb{Z}$ for which there exists some $F'\in \mathcal{F}$ such that for every finite subset $W$ of $F'$,  $m +W \subset F$ for some $m\in \mathbb{Z}$ \cite{li}.
We have $b\mathcal{S}=\mathcal{PS}$, $b\mathcal{F}_{pubd}= \mathcal{F}_{pubd}$ and $b\mathcal{I}^*=\mathcal{T}$. 
Also, by \cite[Lemma 2.1]{li}, $\Delta(\mathcal{F})=\Delta(b\mathcal{F})$.

For $A\subset \Z$, let
$1_A=(x_n)_{n\in\mathbb{Z}}\in \{0,\,1\}^{\mathbb{Z}}$  
with $x_n=1$ whenever $n\in A$.
If $d^*(A)>0$, then  there exists a transitive system  $(Y,\,\sigma),\,Y\subset
 \bar{O}(1_A) $ with an invariant Borel
probability measure $\mu$ with full support  (see comments prior to \cite[Theorem 2.4]{huang3}).  Any TDS with such an invariant measure is called an \emph{$E$-system}.

\begin{thm}\label{strong-scatt}
$b\mathcal{D}=\mathcal{F}_{pubd}$. In particular, $(\mathcal{D}-\mathcal{D})=( \mathcal{F}_{pubd}-\mathcal{F}_{pubd})$ and any set in $(\mathcal{F}_{pubd}-\mathcal{F}_{pubd})$ is an $(IP-IP)$-set.
\end{thm}
\begin{proof}
The proof will be similar to proof in \cite[Example 2.4]{huang-li}.
We have  $\mathcal{D}\subset \mathcal{F}_{pubd}$ and hence
 $b\mathcal{D}\subset b\mathcal{F}_{pubd}=\mathcal{F}_{pubd}$.
 Let $A\in \mathcal{F}_{pubd}$ and set 
   $X=\bar{O}(1_A)$. 
Then $A=N(1_A,\,[1])$. Let  $(Y,\,\sigma)$ be the $E$-system where $Y\subset X$, $Y\neq \{0^{\infty}\}$ and let  $y$ be a
 transitive point in $Y$. 
The fact that $Y$ is an $E$-system, $\bar{O}(y)=Y$ is measure saturated and so $y$  is essentially  recurrent \cite[Theorem 2.6]{ber}.
 Apply \cite[Theorem 2.8]{ber} for $x=y$ and $U_{(y,\,y)}=[1]\times [1]$ to see that 
$$F=\{n\in\mathbb{Z}: (\sigma^ny,\,\sigma^ny)\in [1]\times [1]\}=N(y,\,[1])$$
is a $D$-set. Since $y\in \bar{O}(1_A)$, for any finite subset $W$ of $F$ one has $m+W\subset A$ for some $m\in\mathbb{Z}$. This means that $A\in b\mathcal{D}$ and $b\mathcal{D}=\mathcal{F}_{pubd}$.
\end{proof}


\subsection{Scattering systems}
Theorems~\ref{F^*-mixing} and \ref{filter} show different transitive systems arising from \eqref{diag}. We continue for the remaining cases
in \eqref{diag} using the results and routines from scattering.

Let ${\mathcal U}=\{U_1,\,U_2,\ldots,U_n\}$ be  a finite open cover of
$X$. We call this cover \emph{non-trivial} if $U_i$ is not dense
in $X$ for $1\leq i\leq n$. Let $N(\mathcal U)$ be the least cardinality of a sub-cover
of $\mathcal U$. For this $\mathcal{U}$  and an infinite sequence
$A=\{a_1,\,a_2,\ldots\}\subset {\mathbb N}$ the \emph{complexity
function} of $\mathcal U$ along a sequence $A$
 is defined to be
$${\mathcal C}_A(\mathcal U)=\lim_{n\rightarrow\infty}N(\bigvee_{i=1}^n T^{-a_i}{\mathcal U}).$$
This complexity function is a tool which is used to classify some transitive systems
\cite{huang4,huang3}.


A TDS $(X,\,T)$ is called \emph{full scattering}, \emph{strong scattering}, \emph{scattering} if for any finite non-trivial 
cover $\mathcal{U},\,{\mathcal C}_A(\mathcal U)=\infty$ for $A\in \mathcal{I},\,A\in\mathcal{F}_{pubd}$ and $A\in \mathcal{PS}$ respectively \cite{huang3}. 
There are other characterization for scatterings.
{
 For instance, $(X,\,T)$ is strong scattering if and only if $N(U,\,V)\in 
(\mathcal{F}_{pubd}-\mathcal{F}_{pubd})^*$ and it is scattering if and only if $N(U,\,V)\in (\mathcal{S}-\mathcal{S})^*$ 
where $\mathcal{S}$ is the family of syndetic sets \cite{huang4} and in \cite{huang5} it is shown that $(X,\,T)$ is \emph{weak scattering} if and only if it is $\mbox{Bohr}^*$-transitive.
}

{
It is known that  \cite{huang3}
\begin{eqnarray*}
\text{mixing}&\subsetneq& \!\!\!\text{full scattering }\subseteq \Delta^*-\text{mixing }\subsetneq\text{mild mixing } \subsetneq\text{ weak mixing }\\
&\subsetneq&\!\!\!\text{strong scattering }\subseteq\text{ scattering }\subseteq\text{ weak scattering}\\
&\subsetneq&\!\!\!\text{totally transitive.}
\end{eqnarray*}
There are two sorts of inclutions appearing above, $\subsetneq$ and $\subseteq$ . For $\subseteq$, we actually do not know whether it can be replaced with equality or a proper inclusion.
In fact, the equivalency of full scattering and $\Delta^*$-mixing is an open  problem first raised in \cite{huang3} and the same problem has been questioned  between
 weak scattering and strong scattering  in \cite{dong, huang4}.
}

 Now we   give a sufficient condition for a certain situation  that strong scattering and weak scattering are equivalent.
{ Another case will be dealt in Theorem \ref{C-C-1}.
} 

\begin{thm}\label{strong-scat}
Let $(X,\,T)$ be a weak scattering TDS. If for any opene sets $U,\,V\subset X$
 and Bohr set $B$,
$N(U,\,V)\cap B$ has positive upper Banach density, then
$(X,\,T)$ is strongly scattering.
\end{thm}
\begin{proof}
By the definition of strong scattering, it is sufficient to show
that $N(U,\,V)\cap (A-A)\neq\emptyset$ where $d^*(A)>0$.  For any such 
set $A$, there exists a Bohr set
$B$ and a subset of integers $N$ with $d^*(N)=0$ such that
$B\subset (A-A)\cup N$ \cite[Corollary 5.3]{sumset}. Set
$C:=N(U,\,V)\cap B$ and note that by the assumption $d^*(C)>0$ and so
$C\not\subset N$. This means $(A-A)\cap C\neq\emptyset$ and as a result,
$N(U,\,V)\cap (A-A)\neq\emptyset$.
\end{proof}
 By \cite[  Proof of Theorem 3.1]{huang4},  there is a strong scattering $(X,\,T)$ such that   $N(U,\,V)$ has
zero upper Banach density for some opene sets $U,\, V$. So the necessity does not hold in the above theorem.

Huang and Ye's example in \cite[  Proof of Theorem 3.1]{huang4}
implies that $(\mathcal{IP-IP})^*$-transitives and $(\mathcal{D-D})^*=(\mathcal{F}_{pubd}-\mathcal{F}_{pubd})^*$-transitives are different families. Since when $d^*(N(U,\,V))=0$, then $N(U,\,V)$ is not thick and therefore, $(X,\,T)$ cannot be  mild mixing.

\begin{lem}\label{scat}
Let $(X,\,T)$ be a TDS. Then the following are equivalent.
\begin{enumerate}
  \item $(X,\,T)$ is scattering, that is, $(\mathcal{S-S})^*$-transitive.
  \item $(X,\,T)$ is $(\mathcal{PS}-\mathcal{PS})^*$-transitive.
  \item $(X,\,T)$ is $(\mathcal{C}-\mathcal{C})^*$-transitive.
\end{enumerate}
\end{lem}
\begin{proof}
 From \cite[Lemma 2.1]{li}, $(\mathcal{S-S})^*=(\mathcal{PS-PS})^*$.
 So (1)
$\Leftrightarrow$ (2).
 Also, a set is $PS$ if and only if it is
$C_+$ \cite{ber}. Hence $(\mathcal{C-C})=(\mathcal{PS-PS})$; or, $(\mathcal{C-C})^*=(\mathcal{PS-PS})^*$. Thus (2) $\Leftrightarrow$ (3).
\end{proof}

 Next theorem shows that $\mathcal{D}^*$-mixing, $\mathcal{C}^*$-mixing and $(\mathcal{C-C})^*$-transitive are different. 
\begin{thm}\label{C-C} 
 $\mathcal{C}^*$-mixing is a proper subfamily of  $(\mathcal{D}-\mathcal{D})^*$-transitive.
\end{thm}
\begin{proof}
 By Theorem \ref{strong-scatt}, $(X,\,T)$ is strong scattering if and only if it is $(\mathcal{D}-\mathcal{D})^*$-transitive. 
Also, there is a strong scattering TDS which is not weakly mixing and hence it is not $\mathcal{C}^*$-mixing \cite{huang4}.
 \end{proof}
 Trivially, the above theorem shows that $\mathcal{D}^*$-mixing and $(\mathcal{D-D})^*$-transitive are different families. In fact, in a TDS we have the following inclusions (see \eqref{diag}).  
 \begin{eqnarray*}
\mathcal{D}^*\mbox{-mixing}    \subsetneq \mathcal{C}^*\mbox{-mixing}  &\subsetneq&\mbox{weak mixing}\\
  &\subsetneq&  (\mathcal{D-D})^*\mbox{-transitive}
  \subseteq  (\mathcal{C-C})^*\mbox{-transitive}.
\end{eqnarray*}


\begin{thm}\label{C-C-1} 
An scattering TDS with full support measure is weak mixing.
\end{thm}
\begin{proof}
In \cite{K-Ye}, it has been shown that an MDS $(X,\,\mathcal{B},\,\mu,\,T)$  is weak mixing if and only if $\{n: \mu(A\cap T^{-n}B)>0\}$ is a recurrence set (that is $(\mathcal{C-C})^*$-set)  for any $A,\,B\in \mathcal{B}$. 

Now let $(X,\,T)$  be an scattering TDS with $\mathcal{B}$  its Borel sigma algebra. 
By Lemma \ref{scat}, $(X,\,T)$ is $(\mathcal{C-C})^*$-transitive and with  its full support measure $\mu$,  $(X,\,\mathcal{B},\,\mu,\,T)$ is a weakly mixing MDS and thus a  weakly mixing TDS. 
\end{proof}
\section {Measure dynamical systems and TDS with full support}\label{mds}
{
For an MDS $(X,\,\mathcal{B},\,\mu,\,T)$, let $\mathcal{B}^+=\{B\in\mathcal{B}: \mu(B)>0\}$.  Also for $\epsilon>0$; $A,\,B\in\mathcal{B}^+$     define
 $$N_{\mu}(A,\,B)=\{n\in\mathbb{Z}: \mu(A\cap T^{-n}B)>0\}$$ 
and call
$$R^{\epsilon}_{A,\,B}=\{n\in\mathbb{Z}: \mu(A\cap T^{-n}B)>\mu(A)\mu(B)-\epsilon\}$$
 a \emph{fat intersection} \cite{ber}. 

An MDS is called \emph{$\mathcal{F}$-ergodic} if for any $A,\,B\in\mathcal{B}^+$, $N_{\mu}(A,\,B)\in \mathcal{F}$ and it is called \emph{$\mathcal{F}$-mixing}
 if and only if $R^{\epsilon}_{A,\,B}\in\mathcal{F}$.
{
 In \cite{ber}, it has been shown that 
in spite of different combinatorial families appearing in \eqref{mixingfamily}, the associated MDS families may be the same; namely ${\mathcal D}^*$-mixing  and  ${\mathcal C}^*$-mixing are equal to weak mixing and ${ \Delta}^*$ -mixing is strong mixing.
}
 Now we aim to extend this investigation for  families in \eqref{diag} for MDS.

{
In \cite{K-Ye}, Kuang and Ye show that if 
$\mathcal{F}\neq \Delta$, $(\mathcal{F-F})^*$-ergodic implies $\mathcal{F}^*$-mixing in MDS. For instance 
$(IP-IP)^*$-ergodic is measure-theoretical mild mixing or equivalently $IP^*$-mixing. Also by Theorem \ref{C-C-1},
 $(\mathcal{C-C})^*$-ergodic is $\mathcal{C}^*$-mixinig or equivalently weak mixing. 
But for $\mathcal{F}=\Delta$,  $\Delta^*$-ergodic is not anymore $\Delta^*$-mixing; an example is provided in \cite{K-Ye}. Thus $(\Delta-\Delta)^*$-ergodic is not $\Delta^*$-mixing unlike any other $(F-F)^*$-mixing.
}
Let us summarize our above discussion in Figure \ref{fig1}.  
{
These are in fact the same inclusions as in \eqref{diag}, with three boxes.
First let $\mathcal{F}$ be a family in the middle or the right box, in these situations, $\mathcal{F}$-mixing and $\mathcal{F}$-ergodic are equivalent. 
}
In fact in the middle, they are mild mixing while in the   right they are weak mixing. 
Now for $\mathcal{F}$ in the left box, $\mathcal{F}$-mixings are equivalent and all are strong mixing and whether  $({\Delta}-\Delta)^*$-ergodic is $\Delta^*$-ergodic (which is as  above said different from $\Delta^*$-mixing) is an open problem \cite{K-Ye}.
{
This problem is solved for some important classes of topological dynamics. For instance, 
 in \cite{our-coded} it has been shown that totally transitive coded systems are strong mixing and since all weak mixing are totally transitive so in coded systems all the families are the same.}

\begin{figure}[ht]
\centerline{\includegraphics[width=10cm]{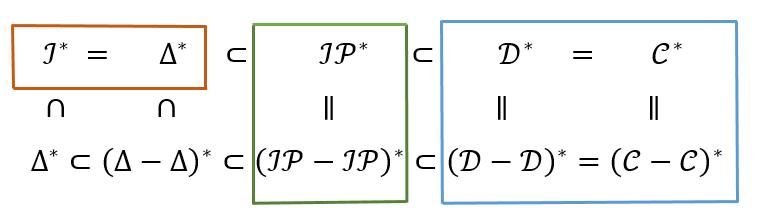}}
\caption{\label{fig1}\small Equivalence properties of Mixing and ergodic in MDS.}
\end{figure}

\subsection{TDS with full support measure and minimal systems}
Systems with full support measure share a few mixing and ergodic properties with their counterparts in MDS with this reservation that ergodic in MDS will be replaced with transitive in TDS. 

{
Thus by the above discussion, in a full support TDS, if $\mathcal{F}\neq \Delta$, then $(\mathcal{F-F})^*$-transitive and $\mathcal{F}^*$-mixing are equivalent.
}
{
Here we briefly examine our above results for systems with full support and recall that an important class of systems with full support is the class of minimal system. 
}

{
%
 
 Moreover, the equivalency of two right boxes in Figure \ref{fig1}  transfer to this case. 
 Since $\Delta^*$-ergodic is a different property from strong mixing in MDS, $\Delta^*$-transitive cannot imply strong mixing. 
 So we have Figure \ref{fig2} for TDS with full support measure from  \eqref{diag}.
 
 \begin{figure}[ht]
\centerline{\includegraphics[width=10cm]{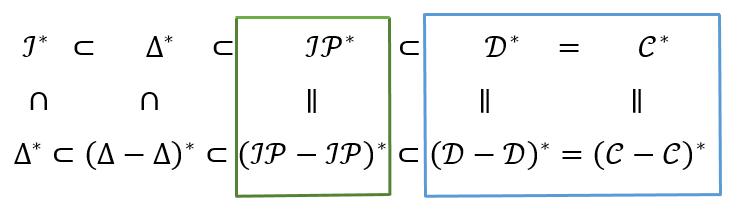}}
\caption{\label{fig2}\small Mixing and transitive properties for TDS with full support measure and minimal systems.}
\end{figure}

{
In minimal systems, we consider the following diagram. Chacon is a minimal mild mixing which is not mixing, so at least one  of the following inclusion must be proper in minimal case.
\begin{eqnarray*}
\mbox{mixing}\subseteq \Delta^*\mbox{-mixing} \subseteq \Delta^*\mbox{-ergodic}  \subseteq (\Delta-\Delta)^*\mbox{-ergodic} \subseteq \mathcal{IP}^*\mbox{-mixing} 
\end{eqnarray*}
 In fact, the equality for the first inclusion was asked in \cite{huang3}; and 
if one can give an example of a minimal $\mathcal{IP}^*$-mixing rigid (mild mixing), then by \cite[Corollary 2.13]{our2} $(\Delta-\Delta)^*\mbox{-ergodic}\subsetneq \mathcal{IP}^*$-mixing. 
}}

}

\end{document}